\documentclass[reqno,12pt]{amsart}
\usepackage{amssymb,amsmath,amsthm,amsxtra,mathrsfs,calc,bm, color}
\usepackage{enumerate}
\usepackage[margin=1in]{geometry}

\usepackage{mathtools}

\def\Z{\mathbb{Z}}

\def\C{\mathbb{C}}

\DeclareMathOperator{\re}{Re}

\newcommand{\pfrac}[2]{\left(\frac{#1}{#2}\right)}

\renewcommand{\(}{\left(}
\renewcommand{\)}{\right)}

\newtheorem{theorem}{Theorem}[section]
\newtheorem{lemma}[theorem]{Lemma}
\newtheorem{corollary}[theorem]{Corollary}
\newtheorem{proposition}[theorem]{Proposition}

\theoremstyle{remark}

\renewcommand{\(}{\left(}
\renewcommand{\)}{\right)}

\def\Z{\mathbb{Z}}
\numberwithin{equation}{section}

\begin{document}


\title{Dissections of strange $q$-series}
\author{Scott Ahlgren, Byungchan Kim, and Jeremy Lovejoy}

\address{Department of Mathematics\\
University of Illinois\\
Urbana, IL 61801} 
\email{sahlgren@illinois.edu}

\address{School of Liberal Arts \\ Seoul National University of Science and Technology \\ 232 Gongneung-ro, Nowongu, Seoul, 01811, Republic of  Korea}
\email{bkim4@seoultech.ac.kr}

\address{Current Address: Department of Mathematics, University of California, Berkeley, 970 Evans Hall \#3780,
Berkeley, CA 94720-3840, USA }

\address{Permanent Address: CNRS, Universit{\'e} Denis Diderot - Paris 7, Case 7014, 75205 Paris Cedex 13, FRANCE}
\email{lovejoy@math.cnrs.fr}

\dedicatory{Dedicated to George E. Andrews on his $80$th birthday}
\date\today
\thanks{The first author was  supported by a grant from the Simons Foundation (\#426145 to Scott Ahlgren). Byungchan Kim was supported by the Basic Science Research Program through the National Research Foundation of Korea (NRF) funded by the Ministry of Education (NRF-2016R1D1A1A09917344)}

\begin{abstract} 
In a study of congruences for the Fishburn numbers, Andrews and Sellers observed empirically that certain polynomials appearing in the dissections of the partial sums of the Kontsevich-Zagier series are divisible by a certain $q$-factorial.    This was proved by the first two authors. In this paper we extend this strong divisibility property to two generic families of $q$-hypergeometric series which, like the Kontsevich-Zagier series, agree asymptotically with partial theta functions.   
\end{abstract}





\maketitle


\section{Introduction}
Recall the usual $q$-series notation
\begin{equation} \label{qfac}
(a;q)_n := (1-a)(1-aq) \cdots (1-aq^{n-1}),
\end{equation}
and let $\mathcal{F}(q)$ denote the Kontsevich-Zagier ``strange'' function \cite{Zag,Zag2},
\begin{equation*}
\mathcal{F}(q) := \sum_{n \geq 0} (q;q)_n.
\end{equation*}
This series does not converge on any open subset of $\mathbb{C}$, but it is well-defined both at roots of unity and   
as a power series when $q$ is replaced by $1-q$.   The coefficients $\xi(n)$ of 
$$
\mathcal{F}(1-q)  = 1+q+2q^2+5q^3 + 15q^4 + 53q^5 + \cdots 
$$
are called the Fishburn numbers, and they count a number of different combinatorial objects (see  \cite{OEIS} for references).

Andrews and Sellers \cite{AS} discovered and proved a wealth of congruences for $\xi(n)$ modulo primes $p$. 
 For example, we have 
\begin{equation} \label{xicongex}
\begin{aligned} 
\xi(5n+4) \equiv \xi(5n+3) &\equiv 0 \pmod{5}, \\
\xi(7n+6) &\equiv 0 \pmod{7}.
\end{aligned}
\end{equation}
In subsequent work of the first two authors, Garvan,  and Straub \cite{AK, Gar, Straub},
similar congruences were obtained for prime powers and for generalized Fishburn numbers.

Taking a different approach, 
Guerzhoy, Kent, and Rolen \cite{GKR} interpreted the coefficients in the asymptotic expansions of  functions $P_{a,b,\chi}^{(1)}(e^{-t})$ defined in \eqref{eq:Pdef} below
in terms of special values of $L$-functions, and  proved congruences for these coefficients using divisibility properties of binomial coefficients.
These congruences are inherited by any function whose expansion at $q=1$ agrees with one of these expansions; these include the function $\mathcal F(q)$ and, 
more generally, the Kontsevich-Zagier functions described in Section~\ref{sec:ex} below.  See \cite{GKR} for details.

Although the  congruences \eqref{xicongex}  bear a passing  resemblance to Ramanujan's congruences for the partition function $p(n)$,
it turns out that they arise from a divisibility property of the partial sums of $\mathcal F(q)$.
For positive integers $N$ and $s$ consider the partial sums 
\begin{equation*}
\mathcal{F}(q;N) := \sum_{n =0}^N (q;q)_n
\end{equation*}
and the $s$-dissection
\begin{equation*}
\mathcal{F}(q;N) = \sum_{i=0}^{s-1} q^iA_s(N,i,q^s).
\end{equation*}
Let  $S(s) \subseteq \{0,1,\dots s-1 \}$ denote the set of reductions modulo $s$ of the set of pentagonal numbers $m(3m+1)/2$, where $m \in \mathbb{Z}$. The key step in the proof of Andrews and Sellers is to show that if $p$ is prime and  $i \not \in S(p)$ then we have
\begin{equation} \label{ASdivis}
(1-q)^n \mid A_p(pn-1,i,q).
\end{equation}
This divisibility property is also important for the proof of the congruences in \cite{Gar,Straub}.
Andrews and Sellers \cite{AS} observed empirically that  $(1-q)^n$ can be strengthened to $(q;q)_n$ in \eqref{ASdivis}. 
The first two authors showed that this divisibility property holds for any $s$.
To be precise, define
\begin{equation}\label{eq:lambdadef}
\lambda(N,s) =\Big \lfloor \frac{N+1}{s} \Big \rfloor.
\end{equation}
Then we have
\begin{theorem}[\cite{AK}] \label{AKthm}
Suppose that $s$ and $N$ are positive integers and  that $i \not \in S(s)$.  Then 
\begin{equation} \label{AKdivis}
(q;q)_{\lambda(N,s)} \mid A_s(N,i,q).
\end{equation}
\end{theorem}

The proof of \eqref{AKdivis} relies on the fact that the Kontsevich-Zagier function satisfies the ``strange identity'' 
\begin{equation*} \label{KZagree}
\mathcal{F}(q) `` = " - \frac{1}{2} \sum_{n \geq 1} n \( \frac{12}{n} \) q^{(n^2 - 1)/24}.
\end{equation*}
Here the symbol $``="$ means that the two sides agree to all orders at every root of unity (this is explained fully in Sections~2  and 5 of  \cite{Zag}).    In this paper we show that a 
analogue of  Theorem \ref{AKthm} holds for a wide class of ``strange'' $q$-hypergeometric series---that is, $q$-series which agree asymptotically with partial theta functions.   

To state our result, let $F$ and $G$ be functions of the form
\begin{align} \label{Habiro_form}
F(q) &= \sum_{n=0}^{\infty} (q;q)_n f_n (q), \\
G(q) &= \sum_{n=0}^{\infty} (q;q^2)_n g_n (q), \label{oHabiro_form}
\end{align}
where $f_n (q)$ and $g_n (q)$ are polynomials.  (Functions of the form \eqref{Habiro_form} are said to lie in the \emph{Habiro ring} \cite{Habiro}.)   Note that $F(q)$ is not necessarily well-defined as a power series in $q$, but it has a power series expansion at every root of unity $\zeta$.   In other words
$F(\zeta e^{-t})$ has a meaningful definition as a  formal power series in $t$ whose coefficients are expressed in the usual way as the 
``derivatives" of $F(\zeta e^{-t})$ at $t=0$.  This is explained in detail in the next section.
Likewise,  $G(q)$  has a power series expansion at every odd-order root of unity. 

We will consider partial theta functions
\begin{equation}\label{eq:Pdef}
P_{a,b,\chi}^{(\nu)}(q) := \sum_{ n \geq 0} n^\nu \chi (n) q^\frac{n^2 - a}b,
\end{equation}
where   $\nu \in \{0, 1\}$, $a\geq 0$ and $b>0$ are  integers, and $\chi : \Z \to \mathbb{C}$ 
is a function satisfying the following properties:
\begin{equation}\label{eq:chi1}
 \chi (n)\neq 0 \quad  \text{only if}\quad  \frac{n^2 - a}b\in\Z,
 \end{equation}
and  for each root of unity $\zeta$,
 \begin{equation}\label{eq:chi2}
	\text{the  function $n\mapsto \zeta^\frac{n^2-a}b\chi(n)$ is periodic and has mean value zero.}
\end{equation}
These assumptions are enough to ensure that for each root of unity $\zeta$, the function 
$P_{a,b,\chi}^{(\nu)} (\zeta e^{-t})$ has an asymptotic expansion as $t\to 0^+$ (see Section~\ref{sec:asymp} below).
We note that  \eqref{eq:chi2} is satisfied by any odd periodic function.  To see this, suppose that $\chi$ is odd with  period $T$,
 and let $\zeta$ be a $k$th root of unity.  Set $M=\operatorname{lcm}(T, b k)$.  Then we have 
 \begin{equation*}
 \zeta^\frac{(M-n)^2-a}b\chi(M-n) =- \zeta^\frac{n^2-a}b\chi(n),
 \end{equation*}
 and so 
 \begin{equation*}
 \sum_{n=0}^{M-1} \zeta^\frac{n^2-a}b\chi(n)=0.
 \end{equation*}

For positive integers $s$ and $N$, consider the partial sum
\begin{equation}\label{eq:fqn}
F(q; N): =\sum_{n=0}^{N}  f_n (q)(q;q)_n 
\end{equation}
and its $s$-dissection
\begin{equation*}
F(q; N) = \sum_{i=0}^{s-1} q^i A_{F, s} (N, i , q^s ).
\end{equation*}
Define 
$S_{a,b, \chi} (s)\subseteq\{0, 1, \dots, s-1\}$ by 
\begin{equation*} \label{eq:Sabdef}
S_{a,b, \chi} (s) := \left\{ \frac{n^2 - a}{b} \pmod{s} \ \ :\ \  \chi(n)\neq 0\right\}.
\end{equation*} 
Our first main result is the following.
\begin{theorem} \label{mainthm}
Suppose that $F$ is a function as in \eqref{Habiro_form} and that $P_{a,b,\chi}^{(\nu)}$ is a function as in \eqref{eq:Pdef}. Suppose that for each root of unity $\zeta$ we have the asymptotic expansion
\begin{equation} \label{F_asym_agree}
P_{a,b,\chi}^{(\nu)} (\zeta e^{-t})\sim F(\zeta e^{-t})\quad \text{as} \quad t \to 0^+.
\end{equation}
Suppose that $s$ and $N$ are positive integers and  that $i \not \in S_{a,b, \chi} (s)$. 
Then we have 
\begin{equation*} 
(q;q)_{\lambda(N,s)} \mid A_{F,s}(N,i,q).
\end{equation*}
\end{theorem}

Analogously, for positive integers $s$ and $N$ with $s$ odd, consider the partial sum
\begin{equation} \label{Habiro_form2}
G(q; N): =\sum_{n=0}^{N}  g_n (q) (q;q^2)_n 
\end{equation}
and its $s$-dissection
\begin{equation*}\label{eq:gqn}
G(q; N) = \sum_{i=0}^{s-1} q^i A_{G,s} (N, i , q^s ).
\end{equation*}
Then the $A_{G,s}(N,i,q^s)$ also enjoy strong divisibility properties. 
Define
\begin{equation}\label{mudefn}
\mu (N,k,s) = \left\lfloor \frac{N}{s(2k-1)} +\frac12 \right\rfloor. 
\end{equation}

\begin{theorem} \label{mainthm2}
Suppose that $G$ is a function as in \eqref{oHabiro_form} and that $P_{a,b,\chi}^{(\nu)}$ is a function as in \eqref{eq:Pdef}.
Suppose  that for each  root of unity $\zeta$ of odd order we have 
\[P_{a,b,\chi}^{(\nu)} (\zeta e^{-t})\sim G(\zeta e^{-t})\qquad \text{as  $t \to 0^+$}.\]
Suppose that $s$ and $N$ are positive integers with $s$ odd and  that $i \not \in S_{a,b, \chi} (s)$. 
Then we have 
\begin{equation*} 
(q;q^2)_{\mu(N,1, s)} \mid A_{G,s} (N,i,q).
\end{equation*}
\end{theorem}

We illustrate Theorem \ref{mainthm2} with an example from Ramanujan's lost notebook.    Consider the $q$-series
\begin{equation*}
\mathcal{G}(q) = \sum_{n \geq 0} (q;q^2)_nq^n.
\end{equation*}
From \cite[Entry 9.5.2]{AnBe} we have the identity
\begin{equation*}
\sum_{n \geq 0} (q;q^2)_nq^n = \sum_{n \geq 0} (-1)^nq^{3n^2+2n} (1+q^{2n+1}),
\end{equation*}
which may be written as
\begin{equation*}
\sum_{n \geq 0} (q;q^2)_nq^n = \sum_{n \geq 0} \chi_{6}(n) q^{(n^2-1)/3},
\end{equation*}
where 
\begin{equation*}
\chi_{6}(n) :=
\begin{cases}
1, &\text{if $n \equiv 1,2 \pmod{6}$}, \\
-1, &\text{if $n \equiv 4,5 \pmod{6}$}, \\
0, &\text{otherwise}.
\end{cases}
\end{equation*}
Therefore,  for each odd-order root of unity $\zeta$ we find that 
\[
P_{1,3, \chi_6}^{(0)} (\zeta e^{-t} ) \sim \mathcal{G}(\zeta e^{-t}) \qquad\text{as $t \to 0^{+}$.}
\]
Since $\chi_{6}$ is odd,   it satisfies conditions \eqref{eq:chi1} and \eqref{eq:chi2}.  Thus, from Theorem~\ref{mainthm2}, we find that for $i \not\in S_{1,3,\chi_6}(s)$ 
 we have  
\begin{equation} \label{Ramadivis}
(q;q^2)_{\lfloor \frac Ns+\frac12 \rfloor} \mid A_{\mathcal{G},s} (N, i, q).
\end{equation}

For example, when $s=5$ we have $S_{1,3,\chi_6}(5) = \{0,1,3\}$.    For $N=8$ we have
\begin{align*}
A_{\mathcal{G},5} (8, 2, q) &= q^2(q;q^2)_2(1+q^2-q^3+2q^4 -q^5+2q^6 + q^8) \\
\intertext{and}
A_{\mathcal{G},5} (8, 4, q) &= -q(q;q^2)_2(1-q+q^2)(1+q + q^2 + q^4 + q^6),
\end{align*}
as predicted by \eqref{Ramadivis}, while the factorizations of $A_{\mathcal{G},5} (8, i, q)$ into irreducible factors for $i \in \{0,1,3\}$ are
\begin{align*}
A_{\mathcal{G},5} (8, 0, q) &= (1-q)(1+q^4-2q^5+q^6 -2q^7+2q^8 -3q^9 + q^{10} - 2q^{11} + q^{12}), \\
A_{\mathcal{G},5} (8, 1, q) &= 1+2q^3-q^4+2q^5-3q^6 + 5q^7 - 5q^8 + 4q^9 - 5q^{10} +4q^{11}-2q^{12} +q^{13} - q^{14},\\
A_{\mathcal{G},5} (8, 3, q) &= q(-1+q^2-2q^3+2q^4-5q^5+5q^6-4q^7+5q^8-4q^9+3q^{10}-2q^{11} + q^{12}).
\end{align*}

The rest of the paper is organized as follows.    In the next section we discuss power series expansions of $F$ and $G$ at roots of unity,
 and in Section~3 we discuss the asymptotic expansions of partial theta functions.   In Section 4 we prove the main theorems.
In Section~5 we give two further examples---one generalizing \eqref{AKdivis} and one generalizing \eqref{Ramadivis}. We close with some remarks on congruences for the coefficients of $F(1-q)$ and
$G(1-q)$.

\section{Power series expansions of $F$ and $G$}
Let $F(q)$ be a function as in \eqref{Habiro_form}   and $G(q)$ be a function as in \eqref{oHabiro_form}. Here we collect some facts which allow us to meaningfully define $F(\zeta e^{-t})$   and $G(\zeta e^{-t})$ as formal power series.
\begin{lemma}\label{stablelemma}
Let  $F(q;N)$ be as in     \eqref{eq:fqn}, and let 
  $G(q;N)$
be as in  \eqref{Habiro_form2}.
Suppose that $\zeta$ is a $k$th root of unity.
\begin{enumerate}
\item 
The values $\(q\tfrac d{dq}\)^{\ell} F(q ;N)\big|_{q=\zeta}$ are stable for $N \geq  (\ell+1)k-1$.
\item If $k$ is odd then the values $\(q\tfrac d{dq}\)^{\ell} G(q ;N)\big|_{q=\zeta}$ are stable for  $2N \geq  (2\ell+1)k$.
\end{enumerate}
\end{lemma}
\begin{proof}
For each positive integer $k$ we have
\[
\begin{aligned}
(1-q^k)^{\ell+1} &\mid (q;q)_N \ \ \ \ \text{for}\ \ \ N \geq (\ell+1)k, \\
(1-q^{2k-1})^{\ell+1} &\mid (q;q^2)_N  \ \ \ \text{for}\ \ \ 2N \geq (2\ell+1)(2k-1) +1.
\end{aligned}
\]
It follows that for $0\leq j \leq \ell$ we have
\[
\begin{aligned}
\(\tfrac d{dq}\)^{j} (q;q)_N\big |_{q =\zeta}&=0\ \ \ \text{for}\ \ \ N \geq (\ell+1)k, \\
\(\tfrac d{dq}\)^{j} (q;q^2)_N\big |_{q=\zeta}&=0\ \ \ \text{for odd $k$ and}\ \ \ 2N \geq (2\ell+1)k+1.
\end{aligned}
\]  
The lemma follows since for any polynomial $f(q)$,  the polynomial $\(q\frac d{dq}\)^\ell f(q)$ is a linear combination 
(with polynomial coefficients) of $\pfrac q {dq}^j f(q)$ with $0\leq j\leq \ell$ (see for example \cite[Lemma 2.2]{AS}).
\end{proof}

For any polynomial $f(q)$, any $\zeta$ and any $\ell\geq 0$ we have  \cite[Lemma 2.3]{AS}
\begin{equation}\label{eq:taylorswap}
\pfrac{d}{dt}^\ell f(\zeta e^{-t})\big|_{t=0}=(-1)^\ell\(q\frac d{dq}\)^\ell f(q)\big|_{q=\zeta}.
\end{equation}

Let $F(q)$ be as in \eqref{Habiro_form} and let $\zeta$ be  a $k$th root of unity.
The last fact together with Lemma~\ref{stablelemma}  allows us to define
\begin{equation*}\label{eq:Fderivdef}
\pfrac d{dt}^\ell F(\zeta e^{-t})\big|_{t=0}:=\pfrac d{dt}^\ell F(\zeta e^{-t}; N)\big|_{t=0}\qquad \text{for any $N\geq k(\ell+1)-1$.}
\end{equation*}

We therefore have a formal series expansion
\begin{equation}\label{eq:formaldiff}
F(\zeta e^{-t})=\sum_{\ell=0}^\infty \frac{\pfrac d{dt}^\ell F(\zeta e^{-t})\big|_{t=0}}{\ell!}\, t^\ell.
\end{equation}

 Similarly, if $G(q)$ is a function as in \eqref{oHabiro_form} and $\zeta$ is a $k$th root of unity with odd $k$, then we can define
\begin{equation}\label{eq:Gderivdef}
\pfrac d{dt}^\ell G(\zeta e^{-t})\big|_{t=0}:=\pfrac d{dt}^\ell G(\zeta e^{-t}; N)\big|_{t=0}\qquad \text{for any $2N\geq k(2\ell+1)$,}
\end{equation}
using \eqref{eq:taylorswap} and Lemma~\ref{stablelemma}. Thus, we have a formal series expansion 
\begin{equation}\label{eq:Gformaldiff}
G(\zeta e^{-t})=\sum_{\ell=0}^\infty \frac{\pfrac d{dt}^\ell G(\zeta e^{-t})\big|_{t=0}}{\ell!}\, t^\ell.
\end{equation}

\section{The asymptotics of $P_{a,b,\chi}^{(\nu)}$}\label{sec:asymp}
In this section we discuss  the asymptotic expansion of the partial theta functions $P_{a,b,\chi}^{(\nu)}(q)$ defined in 
\eqref{eq:Pdef}.  Recall that 
\begin{equation*}
P_{a,b,\chi}^{(\nu)}(q) := \sum_{n \geq 0} n^\nu \chi (n) q^\frac{n^2 - a}b,
\end{equation*}
where   $\nu \in \{0, 1\}$, $a\geq 0$ and $b>0$ are  integers, and $\chi : \Z \to \mathbb{C}$ 
is a function satisfying properties \eqref{eq:chi1} and \eqref{eq:chi2}.

The properties which we describe in the next proposition are more or less standard (see for example \cite[p. 98]{LZ}).
For convenience and completeness we  sketch a proof of the 
 following:
\begin{proposition} \label{asym_prop}
Suppose that $P_{a,b,\chi}^{(\nu)}(q)$ is   as in \eqref{eq:Pdef}.
Let $\zeta$ be a root of unity and let $N$ be  a period of the function $n\mapsto \zeta^\frac{n^2-a}b\chi(n)$.
Then we have the asymptotic expansion 
\[
P_{a,b, \chi}^{(\nu)} ( \zeta e^{-t} ) \sim \sum_{n  = 0}^\infty \gamma_n (\zeta) t^n,\qquad t \to 0^+,
\]
where 
\begin{equation}\label{eq:gamma_form}
\gamma_n (\zeta) = \sum_{\substack{1\leq m\leq N  \\ \chi(m) \neq 0 }} a(m,n, N)\zeta^{\frac{m^2-a}{b}}
\end{equation}
with certain complex numbers $a(m,n,N)$.
\end{proposition}

We begin with a  lemma.  For $n\geq 0$ let  $B_n(x)$ denote the $n$th Bernoulli polynomial.
In the rest of this section we use $s$ for a complex variable since there can be no confusion with the parameter $s$ used above.
\begin{lemma} Let $C:\Z\to\C$ be a function with period $N$ and mean value zero, and let 
\[L(s, C):=\sum_{n=1}^\infty \frac{C(n)}{n^s},\qquad \re(s)>0.\]
Then 
$L(s, C)$ has an analytic continuation to $\C$, and we have 
\begin{equation}\label{eq:lvalues}
L(-n, C)=\frac{-N^n}{n+1}\sum_{m=1}^{N} C(m) B_{n+1}\pfrac mN\qquad\text{for $n\geq 0$}.
\end{equation}
\end{lemma}
\begin{proof}
Let $\zeta(s, \alpha)$ denote the Hurwitz zeta function, whose properties are described for example in \cite[Chapter~12]{AP}.
 We have
\begin{equation}\label{eq:hurwitz}
L(s, C)=N^{-s}\sum_{m=1}^N C(m)\zeta\(s, \tfrac mN\).
\end{equation}
The lemma follows using the fact that each Hurwitz zeta function has only a simple pole with residue $1$ at $s=1$ and
the  formula for the value of each function at $s=-n$ \cite[Thm. 12.13]{AP}.
\end{proof}

\begin{proof}[Proof of Proposition \ref{asym_prop}]
It is enough to prove the proposition for the function 
\[
f(t):= e^{-\frac{at}{b}} P_{a,b,\chi}^{(\nu)}( \zeta e^{-t} ) = \sum_{n \geq 1} n^\nu  \chi (n) \zeta^\frac{n^2 - a}b e^{-\frac{n^2 t }{b}},\qquad \text{ $t>0$}.
 \]
Setting
\begin{equation}\label{eq:cdef}
C(n):= \zeta^\frac{n^2 - a}b\chi (n),
\end{equation}
we have the Mellin transform
\[\int_0^\infty f(t) t^{s-1}\, dt=b^s\Gamma(s) L(2s-\nu, C),\qquad \re(s)>\frac12.\]
Inverting, we find that 
\[f(t)=\frac1{2\pi i}\int_{x=c} b^s\Gamma(s) L(2s-\nu, C)t^{-s}\, ds,\]
for  $c>\frac12$, 
where we write $s=x+i y$.
Using \eqref{eq:hurwitz}, the functional equation for the Hurwitz zeta functions, and
the asymptotics of the Gamma function,  we find that,  for fixed $x$, the function $L(s, C)$ has at most polynomial growth in $|y|$ as $|y|\to\infty$.   Shifting the contour to the line $x=-R-\frac12$  we find that for each $R\geq 0$ we have
\[f(t)=\sum_{n=0}^R\frac{(-1)^n}{b^n n!}L(-2n-\nu, C) t^n+O\(t^{R+\frac12}\),\]
from which
\[f(t)\sim\sum_{n=0}^\infty\frac{(-1)^n}{b^n n!}L(-2n-\nu, C) t^n.\]
The proposition follows from \eqref{eq:cdef} and \eqref{eq:lvalues}.
\end{proof}

\section{Proof of Theorems \ref{mainthm} and \ref{mainthm2}}
We begin with a lemma. The first assertion is   proved in  \cite[Lemma 2.4]{AS},
and the second, which is basically  equation (2.4) in \cite{AK}, follows by extracting an arithmetic progression
using orthogonality. (We note that there is an error in the published version of \cite{AK} which is corrected below;
in that version the operators $\frac d{dq}$ and $q\frac d{dq}$  are conflated in the statement of (2.3) and (2.4).  This does not affect the truth of the rest of the results.)

Let  $C_{\ell,i,j} (s)$ be the array of integers defined recursively as follows:
\begin{enumerate}
\item  $C_{0,0,0} (s) =1$,
\item $C_{\ell,i,0} (s) = i^\ell$ and $C_{\ell,i,j} (s) =0$ for $j \ge \ell+1$ or $j<0$,
\item $C_{\ell+1, i,j} (s) = (i + j s ) C_{\ell,i,j} (s) + s C_{\ell, i, j-1} (s)$ for $1\leq j\leq \ell$.
\end{enumerate}

\begin{lemma}
Suppose that $s$ is a positive integer and that 
\[h(q)=\sum_{i=0}^{s-1}q^i A_s(i, q^s)\]
with polynomials $A_s(i, q)$. Then the following are true:
\begin{enumerate}
\item For all $\ell\geq 0$ we have 
\begin{equation*}\label{diffrec}
\(q\frac{d}{dq}\)^\ell h(q)  =\sum_{j=0}^{\ell}  \sum_{i=0}^{s-1} C_{\ell,i,j} (s) q^{i+js} A_s^{(j)} (i, q^{s}).
\end{equation*}
\item Let $\zeta_s$ be a primitive $s$th root of unity.  Then for   $\ell\geq 0$ and    $i_0\in\{0, \dots, s-1\}$ we have 
\begin{equation}\label{importanteq}
\sum_{j=0}^{\ell}  C_{\ell,i_0,j} (s) q^{i_0+js} A_{s}^{(j)} (i_0 , q^{s})=\frac{1}{s} \sum_{k=0}^{s-1} \zeta_{s}^{-ki_0}   \(  \left( q \tfrac{d}{dq} \right)^\ell h(q) \)\Big|_{q\rightarrow\zeta_{s}^k q}.
\end{equation}
\end{enumerate}

\end{lemma}

\begin{proof}[Proof of Theorem \ref{mainthm}]
Suppose that $F(q)$ and $P_{ a, b, \chi}(q)$ are as in the statement of the theorem.
Suppose that $s$ and $k$ are positive integers, that $i\not\in S_{a,b, \chi}(s)$  and that $\zeta_k$ is a primitive  $k$th root of unity. 
Let  $\Phi_k (q)$ be the $k$th cyclotomic polynomial.
Recall the definition \eqref{eq:lambdadef} of $\lambda(N,s)$ and note that since
\begin{equation} \label{qfaccyclo}
(q;q)_n = \pm \prod_{k=1}^n \Phi_k(q)^{\lfloor \frac{n}{k} \rfloor}
\end{equation}
and
\begin{equation*}
\left \lfloor \frac{\lfloor \frac{x}{s} \rfloor} {k} \right \rfloor= \left \lfloor \frac{x}{ks} \right \rfloor,
\end{equation*}
we have 
\begin{equation*}
(q;q)_{\lambda(N, s) } = \pm \prod_{k=1}^{\lambda(N,s)} \Phi_{k}(q)^{\lambda(N,ks)}.
\end{equation*}
 Therefore, Theorem~\ref{mainthm} will  follow once we show for each $\ell\geq 0$ that 
\begin{equation*}
A_{F,s}^{(\ell)}(N, i, \zeta_k)=0 \ \ \  \text{for}\ \ \  N\geq (\ell+1) ks-1,
\end{equation*}
since this implies that $\Phi_k(q)^{\lambda(N, ks)} \mid A_{F,s} (N, i, q)$ for $1 \leq k \leq \lambda(N,s)$.

From the definition we find  that 
\begin{equation*}\label{basecase}
A_{F,s} (N, i, q ) = \sum_{j=0}^{k-1} q^{j} A_{F, ks} (N, i+js, q^k).
\end{equation*}
If  $i \not\in S_{a,b, \chi}(s)$, then    $i+js \not\in S_{a,b, \chi}(ks)$. It is therefore enough to show that for all $s$, $k$, and $\ell$, and for $i\not\in S_{a,b, \chi}(ks)$, we have
\begin{equation*}
A_{F, ks}^{(\ell)}( N, i , 1) = 0\qquad \text{for $N \geq (\ell+1) ks-1$}.
\end{equation*}
After replacing $ks$ by $s$, it is enough to show that 
for all $s$ and $\ell$, and for $i\not\in S_{a,b, \chi}(s)$, we have
\begin{equation}\label{needtoshowtwo}
A_{F, s}^{(\ell)}( N, i , 1) = 0\qquad \text{for $N \geq (\ell+1) s-1$}.
\end{equation}

We prove \eqref{needtoshowtwo} by induction on $\ell$.  
For the base case $\ell=0$, assume that $N\geq s-1$.
Using \eqref{importanteq} with 
 $q=1$ gives 
\begin{equation*}
 A_{F, s} (N, i , 1)=\frac{1}{s} \sum_{j=0}^{s-1} \zeta_{s}^{-ji}  F(\zeta_s^j; N).
\end{equation*}
By \eqref{F_asym_agree}, \eqref{eq:taylorswap}, Lemma \ref{stablelemma}, and Proposition~\ref{asym_prop} we find that
\[
A_{F, s}(N, i, 1 )=\frac1s\sum_{j=1}^{s} \zeta_{s}^{-ji} \gamma_0(\zeta_{s}^{j}).
\]
By \eqref{eq:gamma_form} and orthogonality (recalling that $i\not\in S_{a,b, \chi}(s)$), we find that $A_{F, s}(N, i, 1 )=0$.

For the induction step, suppose that $N\geq (\ell+1) s-1$, that $i\not\in S_{a,b, \chi}(s)$, and that  \eqref{needtoshowtwo} holds with 
$\ell$ replaced by $j$ for $1 \leq j\leq \ell-1$.
By \eqref{importanteq} and the induction hypothesis we have
\[C_{\ell, i, \ell}(s)A_{F, s}^{(\ell)}(N, i, 1)=
\frac1s \sum_{j=1}^s \zeta_s^{-ji}  \(q\frac d{dq}\)^\ell F (q; N)\big|_{q=\zeta_s^j} .
\]
Using Proposition~\ref{asym_prop}, \eqref{eq:formaldiff},   \eqref{eq:gamma_form}, and orthogonality, we find as above that
\[C_{\ell, i, \ell}(t)A_{F, s}^{(\ell)}(N, i, 1)=0.\]
This establishes \eqref{needtoshowtwo} since $C_{\ell, i, \ell}(s)>0$.  
Theorem~\ref{mainthm} follows.
\end{proof}

\begin{proof}[Proof of Theorem \ref{mainthm2}]
Suppose that $s$ and $k$ are positive integers with $s$ odd, that $i\not\in S_{a,b, \chi}(s)$  and that $\zeta_{2k-1}$ is a   $(2k-1)$th root of unity. Recall the definition \eqref{mudefn} of $\mu(N,k,s)$.  In analogy with  \eqref{qfaccyclo}, we have
\begin{equation*} 
(q;q^2)_n = \pm \prod_{k=1}^n \Phi_{2k-1}(q)^{\lfloor \frac{(2n-1)}{2(2k-1)} + \frac{1}{2} \rfloor},
\end{equation*} 
and as above we obtain
\[
(q;q^2)_{\mu(N,1,s)} = \pm \prod_{k=1 }^{\mu(N,1,s)} \Phi_{2k-1} (q)^{\mu(N,k,s)}.
\]
Therefore, Theorem \ref{mainthm2} follows once we show for each $\ell\geq 0$ that 
\begin{equation*}\label{needtoshow}
A_{G,s}^{(\ell)}(N, i, \zeta_{2k-1})=0 \ \ \  \text{for}\ \ \  2N\geq (2\ell+1) (2k-1)s . 
\end{equation*}

The rest of the proof is similar to that of Theorem \ref{mainthm} (we require $s$ to be odd because  $G(q)$ has a  series expansion only at odd-order roots of unity).
Arguing as above, we show that for each odd $s$ we have 
\[
A_{G,s}^{(\ell)}(N, i, 1)=0 \ \ \  \text{for}\ \ \  2N\geq (2\ell+1)s,
\]
and the result follows.
\end{proof}

\section{Examples}\label{sec:ex}

In this section we illustrate Theorems \ref{mainthm} and \ref{mainthm2} with two families of examples.
\subsection{The generalized Kontsevich-Zagier functions}
In a study of quantum modular forms related to torus knots and the Andrews-Gordon identities, Hikami \cite{Hikami1} defined the functions
\begin{equation}\label{eq:hikami}
X_{m}^{(\alpha)} (q) :=
 \sum_{k_1, k_2,\ldots,k_m \geq 0} (q;q)_{k_m} q^{k_{1}^{2} + \cdots + k_{m-1}^{2} + k_{\alpha+1} + \cdots + k_{m-1}}
\( \prod_{\substack{i=1 \\ i \neq \alpha}}^{m-1} {k_{i+1} \brack k_i} \) {k_{\alpha+1} +1 \brack k_\alpha },
\end{equation}
where $m$ is a positive integer and $\alpha \in \{0,1,\ldots, m-1\}$.  Here we have used the usual $q$-binomial coefficient (or Gaussian polynomial)
\begin{equation*}
{n \brack k} := {n \brack k}_q :=
\begin{cases}
\frac{(q;q)_{n}}{(q;q)_{k}(q;q)_{n-k}}, & \text{if $0 \leq k \leq n$},\\ 
0, & \text{otherwise}.
\end{cases}
\end{equation*} 
The simplest example
\begin{equation*}
X_{1}^{(0)} (q)= \sum_{n \geq 0} (q;q)_{n}
\end{equation*} 
is the Kontsevich-Zagier function. 
From \eqref{eq:hikami} we can write  
\begin{equation*}
X_{m}^{(\alpha)} (q) = \sum_{k_m \geq 0} (q;q)_{k_m} f_{k_m}^{(\alpha)} (q),
\end{equation*}
with polynomials $f_{k_m}^{(\alpha)} (q)$.

Hikami's identity \cite[eqn (70)]{Hikami1} implies that for each root of unity $\zeta$ we have
\[
P_{(2m-2\alpha-1)^2 , 8(2m+1) , \chi_{8m+4}^{(\alpha)}}^{(1)} (\zeta e^{-t}) \sim  X_{m}^{(\alpha)} (\zeta e^{-t})
\]
as $t \to 0^{+}$, where $\chi_{8m+4}^{(\alpha)} (n)$ is defined by
\begin{equation}\label{eq:chimdef}
\chi_{8m+4}^{(\alpha)} (n) =
\begin{cases}
-1/2, &\text{if $n \equiv 2m-2\alpha-1$ or $6m+2\alpha+5 \pmod{8m+4}$,} \\
1/2, &\text{if $n \equiv 2m+2\alpha+3$ or $6m-2\alpha+1 \pmod{8m+4}$,} \\
0, &\text{otherwise.}
\end{cases}
\end{equation}
The function  $\chi_{8m+4}^{(\alpha)} (n)$ satisfies condition \eqref{eq:chi1}.
For  
 \eqref{eq:chi2} we record a short lemma.
\begin{lemma}\label{lem:genchar}
Suppose that $\chi_{8m+4}^{(\alpha)} (n)$ is as defined in \eqref{eq:chimdef} and that $\zeta$ is a root of unity of order $M$.
Define
\[\psi(n)=\zeta^\frac{n^2-(2m-2\alpha-1)^2}{8(2m+1)}\chi_{8m+4}^{(\alpha)} (n).\]
Then 
\[\sum_{n=1}^{M(8m+4)}\psi(n)=0.\]
\end{lemma}
\begin{proof}
Note that $\psi$ is supported on odd integers, so we assume in what follows that $n$ is odd. From the definition, we have
\begin{equation}\label{eq:chisign}
\chi_{8m+4}^{(\alpha)}(n+M(4m+2))=(-1)^M\chi_{8m+4}^{(\alpha)}(n).
\end{equation}
The exponent in the ratio of the corresponding powers of $\zeta$ is $mM^2+\frac{M^2+Mn}2$.
So the ratio of these powers of $\zeta$ is 
\[\zeta^\frac{M^2+Mn}2.\]
If $M$ is odd then this becomes $\zeta^{M\pfrac{M+n}2}=1$, while if $M$ is even then this 
becomes $\zeta^\frac{M^2}2\zeta^{\frac M2 n}=-1$ (since $M$ is the order of $\zeta$ and $n$ is odd).
Therefore the ratio in either case is $(-1)^{M+1}$.  Combining this with \eqref{eq:chisign} gives 
\[\psi(n+M(4m+2))=-\psi(n),\]
from which the lemma follows.
\end{proof}

  Therefore $X_{m}^{(\alpha)} (q)$ satisfies the conditions of Theorem \ref{mainthm}, and we obtain the following.
\begin{corollary} \label{Hikami_div}
If $s$ is a positive integer and  $i \not\in S_{(2m-2\alpha-1)^2 , 8(2m+1) , \chi_{8m+4}^{(\alpha)}} (s)$, then
 \[(q;q)_{\lambda(N, s)}\big |  A_{X_{m}^{(\alpha)}, s} (N, i, q),  \]
where $A_{X_{m}^{(\alpha)}, s} (N, i, q)$ are the coefficients in the  $s$-dissection of the partial sums (in $k_m$) of $X_{m}^{(\alpha)}(q)$. 
\end{corollary}

For example, when $s=3$ we have $S_{9,40,\chi_{20}^{(0)}}(3) = \{0,1\}$ and $S_{1,40,\chi_{20}^{(1)}}(3) = \{0,2\}$. For $N=8$ we have
\begin{align*}
A_{X_{2}^{(0)},3} (8, 2, q) &= (q;q)_3 (1+q)(1+q+q^2)(1-q+\cdots -q^{25}+q^{26}) \\
\intertext{and}
A_{X_{2}^{(1)},3} (8, 1, q) &= (q;q)_3 (1+q)(1-q+q^2)(1+q+q^2)(1+2q+\cdots -q^{26}+q^{27}),
\end{align*}
as predicted by Corollary~\ref{Hikami_div}, while 
\begin{align*}
A_{X_{2}^{(0)},3} (8, 0, q) &=(1-q+q^2)(9+9q+\cdots+q^{33}+q^{34}), \\
A_{X_{2}^{(0)},3} (8, 1, q) &=-8-7q+\cdots + q^{34} - q^{35}, \\
A_{X_{2}^{(1)},3} (8, 0, q) &=9-7q+\cdots+2q^{36}+q^{39}, \\
\intertext{and}
A_{X_{2}^{(1)},3} (8, 2, q) &= -7+3q^3 -\cdots +q^{36}-q^{38}\\
\end{align*}
are not divisible by $(q;q)_{3}$.

\subsection{An example with $\nu=0$}
For $k\geq 1$
let $\mathcal{G}_k(q)$ denote the $q$-series
\begin{equation*}
\mathcal{G}_k(q) = \sum_{n_k \geq  n_{k-1} \geq \cdots \geq n_1 \geq 0}q^{n_k+2n_{k-1}^2+2n_{k-1} + \cdots + 2n_1^2 + 2n_1}(q;q^2)_{n_k}\begin{bmatrix} n_k \\ n_{k-1} \end{bmatrix} _{q^2} \cdots \begin{bmatrix} n_2 \\ n_1 \end{bmatrix} _{q^2}.
\end{equation*}
Then we have the identity
\begin{equation} \label{G_k1}
\mathcal{G}_k(q) = \sum_{n \geq 0} (-1)^nq^{(2k+1)n^2+2kn}(1+q^{2n+1}),
\end{equation}
which follows from Andrews' generalization \cite{And} of the Watson-Whipple transformation 
\begin{equation*} \label{Andrews}
\begin{aligned}
\sum_{m = 0}^N & \frac{(1-aq^{2m})}{(1-a)}\frac{(a,b_1,c_1,\dots,b_k,c_k,q^{-N})_m}{(q,aq/{b_1},aq/{c_1},\dots,aq/{b_k},aq/{c_k},aq^{N+1})_m}\left(\frac{a^kq^{k+N}}{b_1c_1\cdots b_kc_k}\right)^m \\
&= \frac{(aq,aq/{b_kc_k})_N}{(aq/{b_k},aq/{c_k})_N} \sum_{N \geq n_{k-1}\geq \cdots \geq n_1 \geq 0} \frac{(b_k,c_k)_{n_{k-1}}\cdots(b_2,c_2)_{n_1}}{(q;q)_{n_{k-1}-n_{k-2}}\cdots (q;q)_{n_2-n_1}(q;q)_{n_1}} \\
&\phantom{\times} \hskip1in \qquad \times \frac{(aq/b_{k-1}c_{k-1})_{n_{k-1}-n_{k-2}}\cdots(aq/b_2c_2)_{n_2-n_1}(aq/b_1c_1)_{n_1}}{(aq/b_{k-1},aq/c_{k-1})_{n_{k-1}}\cdots(aq/b_1,aq/c_1)_{n_1}} \\
&\qquad \qquad \qquad \qquad \qquad \qquad \times \frac{(q^{-N})_{n_{k-1}}(aq)^{n_{k-2}+\cdots+n_1}q^{n_{k-1}}}{(b_kc_kq^{-N}/a)_{n_{k-1}} (b_{k-1}c_{k-1})^{n_{k-2}}\cdots (b_2c_2)^{n_1}}.
\end{aligned}
\end{equation*}
Here we have extended the notation in \eqref{qfac} to
\begin{equation*}
(a_1,a_2,\dots, a_k)_n := (a_1;q)_n(a_2;q)_n \cdots (a_k;q)_n.
\end{equation*}
To deduce \eqref{G_k1}, we set $q=q^2,a=q^2, b_k=q$, and $c_k = q^2$ and then let $N \to \infty$ along with all other $b_i,c_i$.

The identity \eqref{G_k1} may be written as
\begin{equation*}
\mathcal{G}_k(q) = \sum_{n \geq 0} \chi_{4k+2}(n)q^{\frac{n^2-k^2}{2k+1}},
\end{equation*}
where 
\begin{equation*}
\chi_{4k+2}(n) :=
\begin{cases}
1, &\text{if $n \equiv k,k+1 \pmod{4k+2}$}, \\
-1, &\text{if $n \equiv -k,-k-1 \pmod{4k+2}$}, \\
0, &\text{otherwise}.
\end{cases}
\end{equation*}
This implies that for each odd-order root of unity $\zeta$, we have
\[
P_{k^2,2k+1, \chi_{4k+2}}^{(0)} (\zeta e^{-t} ) \sim G_{k} (\zeta e^{-t}) \qquad\text{as $t \to 0^+$}.
\]
The function $\chi_{4k+2}(n)$ satisfies   conditions \eqref{eq:chi1} and \eqref{eq:chi2}
(see the remark following \eqref{eq:chi2}),
so    Theorem \ref{mainthm2} gives
\begin{corollary}
Suppose that $k$ and $N$ are positive integers, that $s$ is  a positive odd integer, and that $i \not\in S_{k^2,2k+1, \chi_{4k+2}} (s) $.
Then
\[
(q;q^2)_{\lfloor \frac Ns+\frac12 \rfloor} \mid A_{\mathcal{G}_k, s} (N, i, q).
\]
\end{corollary}

\section{Remarks on congruences} \label{sec:cong}
Congruences for the coefficients of the functions $F(q)$ and $G(q)$ in Theorems~\ref{mainthm} and \ref{mainthm2} can be deduced from the results of \cite{GKR}.
In closing we mention another approach. Theorems~\ref{mainthm} and \ref{mainthm2} guarantee that many of the coefficients in the $s$-dissection are divisible by high powers of $1-q$,
and the congruences follow from this fact when $s=p^r$ together  with an argument as in \cite[Section 3]{AK}.

For example, let $\mathcal G_k$ be the function defined in the last section and 
 define $\xi_{\mathcal{G}_k}(n)$ by
\begin{equation*}
\mathcal{G}_k (1-q) = \sum_{n \geq 0} \xi_{\mathcal{G}_k} (n) q^n.
\end{equation*}
Consider the expansions 
\begin{align*}
\mathcal{G}_1(1-q) &= \sum_{n \geq 0} \xi_{\mathcal{G}_1}(n)q^n = 1 + q+ 2q^2 + 6q^3 + 25q^4 + 135q^5 + \cdots, \\
\mathcal{G}_2 (1-q) &= \sum_{n \geq 0} \xi_{\mathcal{G}_2} (n) q^n = 1 + 2 q + 6 q^2 + 28 q^3 + 189 q^4 + 1680 q^5 + \cdots.
\end{align*}
Then we have such congruences as
\begin{align*}
\xi_{\mathcal{G}_1}(5^r n-1) &\equiv 0 \pmod{5^r}, \\
\xi_{\mathcal{G}_1}(7^rn-1) &\equiv 0 \pmod{7^r}, \\
\xi_{\mathcal{G}_1}(13^rn-\beta) &\equiv 0 \pmod{13^r}
\end{align*}
for $\beta \in \{1, 2, 3, 4\}$, and 
\begin{align*}
\xi_{\mathcal{G}_2} (7^rn-1) &\equiv 0 \pmod{7^r}, \\
\xi_{\mathcal{G}_2} (11^rn-1) &\equiv 0 \pmod{11^r}.
\end{align*}

\end{document}